\documentclass[12pt,a4paper]{article}

\usepackage[english]{babel}
\usepackage[T1]{fontenc}
\usepackage[utf8]{inputenc}
\usepackage{amsthm}
\usepackage{amsmath}
\usepackage{graphicx}
\usepackage{amssymb}
\usepackage{latexsym}
\usepackage{color}
\usepackage{url}
\usepackage{subcaption}
\usepackage{titlesec}

\theoremstyle{definition}
\newtheorem{definition}{Definition}[section]
\theoremstyle{theorem}
\newtheorem{theorem}[definition]{Theorem}
\newtheorem{corollary}[definition]{Corollary}
\newtheorem{lemma}[definition]{Lemma}
\newtheorem{conjecture}{Conjecture}
\newtheorem{remark}[definition]{Remark}
\theoremstyle{definition}
\newtheorem{problem}[conjecture]{Problem}

\newcommand{\K}{\mathbb{K}}

\newcommand{\TTT}{\operatorname{TTT}}
\newcommand{\TTTsc}{\operatorname{TTT}_{\text{sc}}}

\newcommand{\col}{\operatorname{Col}}
\newcommand{\superp}{\operatorname{Sup}}

\newcommand{\klein}{\mathbb{Z}_2 \times \mathbb{Z}_2}

\let\phi\varphi
\def\phisum{\phi_{*}}

\titleformat*{\paragraph}{\itshape}

\begin{document}


\title{Strictly critical snarks with girth\\ or cyclic connectivity equal to 6}

\author{
Ján Mazák, Jozef Rajník, Martin Škoviera
\\[3mm]
\\{\tt \{mazak, rajnik, skoviera\}@dcs.fmph.uniba.sk}
\\[5mm]
Comenius University, Mlynská dolina, 842 48 Bratislava\\
}

\maketitle

\begin{abstract}
A snark -- connected cubic graph with chromatic index $4$ -- is critical if the graph resulting from the removal of any pair of distinct adjacent vertices is $3$-edge-colourable; it is bicritical if the same is true for any pair of distinct vertices. A snark is strictly critical if it is critical but not bicritical. Very little is known about strictly critical snarks. Computational evidence suggests that strictly critical snarks constitute a tiny minority of all critical snarks. Strictly critical snarks of order $n$ exist if and only if $n$ is even and at least 32, and for each such order there is at least one strictly critical snark with cyclic connectivity $4$. A sparse infinite family of cyclically $5$-connected strictly critical snarks is also known, but those with cyclic connectivity greater than $5$ have not been discovered so far. In this paper we fill the gap by constructing cyclically $6$-connected strictly critical snarks of each even order $n\ge 342$. In addition, we construct cyclically $5$-connected strictly critical snarks of girth 6 for every even $n\ge 66$ with $n\equiv 2\pmod8$.
\end{abstract}





\bigskip

\section{Introduction}
Cubic graphs that do not admit any proper edge colouring with
three colours, known as \emph{snarks}, play a crucial role in
the study of a variety of problems related to
flows, edge colourings, perfect matchings, or cycle covers of
graphs. 
Over time, considerable effort has been exerted to
understand the structure of snarks. Although several
relevant partial results have been gathered, many deep questions
remain open.


In our recent paper, \emph{Morphology of small snarks}
\cite{Morphology}, we have analysed the structure of all snarks
with at most 36 vertices.
The cornerstone of our approach to structural analysis
is the concept of
an \emph{{\rm I}-extension}.
This operation, also known as an \emph{edge extension}, consists in specifying
two edges in a cubic graph, subdividing each of them with a new vertex, and
adding a new edge joining the two vertices. It is a very
natural operation and in the literature it has been used many
times, see for example \cite{AFJ, BGHM, McCuaig, Wormald}.

Edge extensions can be conveniently employed to construct new
snarks from known ones: it is enough to take an existing
snark, choose a \emph{removable} pair of edges (that is, one that can be removed without
breaking uncolourability), and perform an I-extension. Removable
pairs of edges are present in most known snarks; actually,
an overwhelming majority of known snarks can be obtained from a
smaller snark by a series of I-extensions retaining
uncolourability at each step. However, there exist snarks that
cannot be obtained from a smaller snark by an I-extension ---
and these are exactly the critical snarks. A \emph{critical} snark
can be equivalently defined as one in which the removal of any
two adjacent vertices produces a $3$-edge-colourable graph. If
the removal of any pair of vertices  yields a 
$3$-edge-colourable graph, a snark is called \emph{bicritical}.
A snark which is critical but not bicritical is called
\emph{strictly critical}.

Critical snarks have been emerging in
the literature under different disguises and in different
settings \cite{dSPL, dVNR, Nedela, S-full}. In 1996, Nedela and Škoviera \cite{Nedela} introduced
critical and bicritical snarks in the context of snark
reductions and decompositions. They showed that critical snarks
coincide with $6$-irreducible snarks and that bicritical snarks
are $7$-irreducible, which means that they are the same as
the irreducible ones. In 2008, da Silva et al.
\cite{dSL, dSPL} initiated the study of flow criticality of graphs and introduced flow-edge-critical and
flow-vertex-critical snarks.
It turned out, however, that these two approaches to snark criticality agree (Máčajová and
Škoviera \cite{Macajova21}): a snark is $4$-flow-edge-critical
if and only if it is critical, and is $4$-flow-vertex-critical
if and only if it is bicritical.


Critical snarks are known to be cyclically $4$-connected with
girth at least $5$ \cite{Nedela} and therefore they are nontrivial snarks by the usual standards. As discussed in
\cite{Chladny-Skoviera-Factorisations}, every snark $G$ with
cyclic connectivity $4$ can be constructed as a dot product $G_1\cdot G_2$ of
two smaller snarks. If $G$ is bicritical, both $G_1$ and $G_2$ are bicritical. As a consequence, every bicritical snark can be decomposed
into a collection of cyclically $5$-connected bicritical
snarks such that it can be reconstructed from them by repeated
dot products. Moreover, this collection is unique up to isomorphism \cite[Theorem~C]{Chladny-Skoviera-Factorisations}.
The class of bicritical snarks is thus closed under such decomposition.
If $G$ is strictly critical, then $G_2$ must be critical but
$G_1$ need not; it only has to be ``nearly critical''  \cite[Sections~4 and~6]{Chladny-Skoviera-Factorisations}. According to \cite[Section~12]{Chladny-Skoviera-Factorisations}, snarks arising from the decomposition of strictly critical snarks might even include snarks with cyclic connectivity 3. Overall, very little is understood about the relationship between strictly critical snarks and ``nearly critical'' snarks, or about the nature of graphs in either of these two classes.

Strictly critical snarks have been previously studied in \cite{Chladny-Skoviera-Factorisations,Grunewald,Morphology, SteffenStrictlyCritical}.
The smallest strictly critical snarks have order $32$ (one is displayed in Figure~\ref{fig:scrit-32}) 
and such snarks have been constructed for all higher orders.
\begin{theorem}[\cite{Chladny-Skoviera-Factorisations}]
There exists a strictly critical snark of order $n$ for each even $n\ge 32$.
\label{thm:conn4}
\end{theorem}
\begin{figure}
    \centering
    \includegraphics{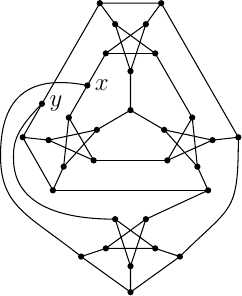}
    \caption{A strictly critical snark of the smallest possible order $32$ with the only removable pair $\{x, y\}$ of vertices.}
    \label{fig:scrit-32}
\end{figure}
There are exactly 64~326~024 cyclically $4$-connected snarks of girth at least 5 with order not exceeding 36, of which 55~172 are critical \cite{BCGM}. Somewhat surprisingly, a vast majority of these snarks are bicritical, only 846 being strictly critical,
a little over 1.5 percent \cite{BCGM, Macajova21}.
This fact does not seem to have any obvious explanation.

All strictly critical snarks discovered so far have cyclic
connectivity $4$ or $5$. As already mentioned, the smallest strictly critical snark has order $32$. Its cyclic connectivity is $4$. The smallest cyclically $5$-connected strictly critical snarks have order 36 (and there are 84 such snarks \cite[Section~7]{Morphology}). It is not known how common cyclically $5$-connected strictly critical snarks are. An infinite family of them was provided by Gr\"unevald and Steffen \cite{Grunewald}, but the graphs in the family are rather large (the smallest of them has $66$ vertices) and only a small proportion of orders of snarks is covered.

In Section~\ref{sec:sc-6-6}, we demonstrate the existence of strictly critical snarks with cyclic connectivity $6$, proving an analogue of Theorem~\ref{thm:conn4} and solving Problem~6.3 from \cite{Chladny-Skoviera-Factorisations}.

\begin{theorem}\label{thm:main}
There exists a cyclically $6$-connected strictly critical snark 
of order $n$ for each even $n\ge 342$.
\end{theorem}

We discuss the situation for orders below $342$ at the end of Section~\ref{sec:sc-6-6} (see Theorem~\ref{thm:sc-cc6} and Problem~\ref{problem1}).

\medskip

In addition to cyclically 6-connected strictly critical snarks we present a rich family of cyclically 5-connected strictly critical snarks with girth 6. These snarks are considerably smaller than the discovered snarks with cyclic connectivity 6. 

\begin{theorem}
There exists a cyclically $5$-connected strictly critical snark with girth $6$ and order $n$ for each $n\ge 66$ such that $n\equiv 2\pmod 8$.
\label{thm:cyc5}
\end{theorem}

We prove this theorem in Section~\ref{sec:cc5}. The construction is inspired by specimens found among critical snarks of order $36$ and analysed in Section~7 of \emph{Morphology} \cite{Morphology}. Our present work can thus be considered as a supplement to Section~8 of \cite{Morphology} where we have generalised small examples of snarks with desirable properties into abundant families of bicritical snarks.

Snarks with cyclic connectivity or girth greater than 5 are interesting for several reasons. First, cyclic connectivity and girth are important measures of complexity of cubic graphs.
For example, smallest cubic graphs of given girth, called \emph{cages}, have been studied for decades.
Second, these two parameters have strong influence on structural properties of cubic graphs, including their edge-colourability. Some 40 years ago it was conjectured \cite{Jaeger} that snarks with cyclic connectivity or girth greater than 6 do not exist. It took 16 years to refute the girth conjecture \cite{Superposition}, while the conjecture on cyclic connectivity remains open and is essentially intact. It is thus interesting, for every property of a snark, to ask whether there are cyclically 6-connected snarks having that property --- we might discover that some nontrivial property does not hold for snarks of higher cyclic connectivity and that might shed some light on the Jaeger's connectivity conjecture. It has also been conjectured \cite[Conjecture~1]{Nedela} that bicritical snarks have girth at most~6. We might learn more about this problem by trying to construct strictly critical and bicritical snarks with higher girth.
Finally, the third reason for our interest in cyclically $6$-connected strictly critical snarks is their complete absence among known snarks with at most 38 vertices, which distinguishes them from bicritical ones.

\section{Preliminaries}

\subsubsection*{Basics}
All graphs considered in this paper are undirected; they may contain loops and parallel edges, although their presence will often be excluded by connectivity requirements. Recall that a cubic graph is cyclically \emph{$k$-connected} if no set of fewer than $k$ edges separates two cycles of $G$ from each other. The \emph{cyclic connectivity} of $G$ is the largest integer $k$ such that $G$ is cyclically $k$-connected.

\subsubsection*{Multipoles}

Multipoles are often used as building blocks for large graphs. Informally, a \emph{multipole} is a graph permitted to contain dangling edges and isolated edges. Each edge has two ends, which may but need not be incident with a vertex. A \emph{dangling edge} has only one end incident with a vertex while \emph{isolated edge} has neither.
A \emph{link} is an edge with both ends incident with a vertex.
An edge end incident with no vertex is called a \emph{semiedge} and a multipole with $k$ semiedges is a \emph{$k$-pole}. The set of all semiedges of $M$ is denoted by $S(M)$. The semiedges of a multipole $M$ can be grouped into pairwise disjoint \emph{connectors} $S_1, \dots, S_t$ (each connector comes with a linear ordering of its semiedges). Such multipole is called a $(|S_1|, \dots, |S_t|)$-pole and is denoted by $M(S_1,\dots,S_t)$.  In this article, we only deal with cubic multipoles, which means that each vertex is incident with exactly three edge ends.

The \emph{width} of a connector is the number of semiedges it contains. A connector of width $k$ is called a \emph{$k$-connector}. The \emph{junction} of two semiedges is performed by joining the semiedges together, creating an ordinary edge from two dangling edges. We can also perform a junction of two connectors of the same width or two $k$-poles by performing individual junctions for each pair of semiedges (this is the place where the linear ordering of semiedges defined for each connector comes into play). 

Formal definitions of these notions and other related information can be found in \cite{Morphology}.

For a $k$-pole $M(S_1, S_2, \dots, S_t)$ and its vertex $v$ we denote by $M - v$ a $(k+3)$-pole $N(S_1, S_2, \dots, S_t, S_{t+1})$ constructed by removing the vertex $v$ from the multipole $M$ and putting the three semiedges formerly incident with $v$ into the connector $S_{k+1}$. Note that if the vertex $v$ is incident with a dangling edge $e$, then $e$ becomes an isolated edge and the ends of $e$ are retained in the multipole $M - v$. To keep our notation short, we shall write $M - (v_1,v_2,\dots,v_r)$ instead of $(((M - v_1) - v_2) - \dots )- v_r$. Note that this operation is not commutative: the order of the connectors changes with the order of removed vertices. When $u$ and $v$ are adjacent vertices of a $(k_1, \dots, k_t)$-pole $M$, we denote $M - [u, v]$ the $(k_1, \dots, k_t, 2, 2)$-pole obtained from the $(k_1, \dots, k_t, 3,3)$-pole $M - (u, v)$ by deleting the isolated edge obtained from the edge $uv$ of $M$.

Assume that $e$ is a link of a multipole $M(S_1, \penalty0 S_2, \dots, S_k)$. By cutting $e$ into two dangling edges ending with semiedges $f_1$ and $f_2$, respectively, we construct a new multipole $N(S_1, S_2, \dots, S_{k}, (f_1, f_2))$ which we denote by $M - e$. Again, we write $M~-~(e_1, e_2, \dots, e_s)$ instead of $(((M - e_1) - e_2) - \dots )- e_s$. Finally, for a vertex $v$ of a $(k + 1)$-pole $M$ incident with exactly one dangling edge $e$, we denote by $M \sim v$ the $k$-pole obtained from $M$ by deleting the dangling edge $e$ and by suppressing the vertex $v$.

\subsubsection*{Colourings and flows}

A \emph{$3$-edge-colouring} of a multipole $M$, or just a \emph{colouring} for short, is an assignment of colours to the edges of $M$ such that no two edges incident with the same vertex receive the same colour. We extend the notion of a colouring to the ends of edges in a natural fashion: an end of an edge (in particular, a semiedge) has the same colour as the edge it belongs to. Both ends of an edge thus always have the same colour.

It has become standard to take $\K = \mathbb{Z}_2 \times \mathbb{Z}_2 - \{(0, 0)\}$ for the set of colours.
With this choice, a mapping $\varphi\colon E(M)\to\K$ 
is a colouring if and only if, for any three edges $e$, $f$, and $g$ incident with the same vertex, one has
$\varphi(e) + \varphi(f) + \varphi(g) = 0$. This means that a colouring of a 
multipole is a nowhere-zero $(\klein)$-flow and vice versa. A straightforward consequence of this fact is the following lemma.

\begin{lemma}[Parity Lemma \cite{Descartes}]\label{lemma:parity}
	Let $M$ be a $k$-pole and let $k_1$, $k_2$, $k_3$ be the
	numbers of dangling edges of colour $(0,1)$, $(1,0)$, $(1,1)$, respectively. Then
	$$ k_1 \equiv k_2 \equiv k_3 \equiv k \pmod{2}.$$
\end{lemma}

For any ordered set $A = (e_1, e_2, \dots, e_k)$ of semiedges
(in particular, a connector) we denote $(\phi(e_1), \phi(e_2),
\dots, \phi(e_k))$ by $\phi(A)$. The \emph{colouring set} of a
multipole $M$ is the set
$$\col(M) = \{ \phi(S(M)) \mid \phi \text{ is a colouring of } M\}.$$
Two multipoles $M_1$ and $M_2$ are said to be \emph{colour-equivalent} if $\col(M_1) = \col(M_2)$.

The \emph{flow through a connector $S$} of $M$ is the value
$$\phisum(S) = \sum_{e \in S} \phi(e).$$ 

A connector $S$ of $M$ is called \emph{proper} if $\phisum(S) \ne 0$ for each colouring $\phi$ of $M$. If all connectors of a multipole $M$ are proper, we say that $M$ itself is \emph{proper}. Note that every uncolourable multipole is proper. 

\section{Cyclically 5-connected strictly critical snarks of girth~6}
\label{sec:cc5}

In this section we prove Theorem~\ref{thm:cyc5} by generalising a structure discovered in cyclically $5$-connected strictly critical snarks of order 36 and described in \cite[Section~7]{Morphology}. The structure can be built from three $(2,3)$-poles as follows (see Figure~\ref{fig:ttt-sc}).

\begin{definition}
    For three $(2,3)$-poles $T_1$, $T_2$, and $T_3$, let $\TTTsc(T_1, T_2, T_3)$ denote the $(2,2,2)$-pole constructed from the disjoint union 
    $T_1\cup T_2 \cup T_3$
    by adding three new vertices $v_1$, $v_2$, and $v_3$ and by attaching the dangling edges from the $3$-connectors to them in such a way that each $v_i$ becomes incident with exactly one edge from each $3$-connector. 
\end{definition}

\begin{figure}[h]
	\centering
	\includegraphics{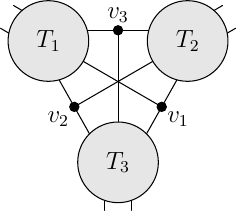}
	\caption{A $(2,2,2)$-pole found in cyclically $5$-connected strictly critical snarks of order 36}
	\label{fig:ttt-sc}
\end{figure}


In order to produce a snark containing a $(2,2,2)$-pole $\TTTsc$ we require the three 
$(2,3)$-poles $T_1$, $T_2$, and $T_3$ be proper. It is well known and easy to see that if $G$ is a snark, $e$ is an edge of~$G$, and $v$ is a vertex of $G$, then the $(2, 3)$-pole 
$T(B, C) = (G - e) - v$ is proper. It means that for each colouring $\varphi$ of $T$ we have
\begin{equation}
    \phisum(B) \neq 0 \qquad \hbox{and} \qquad \phisum(C) \neq 0.
    \label{eq:proper23}
\end{equation}
If $T$ admits all colourings satisfying both \eqref{eq:proper23} and Parity Lemma, it is called  \emph{perfect}.  To fulfil the connectivity requirements of the resulting snarks it is convenient to assume that the $(2,3)$-pole $(G - e) - v$ is created by using pairs $e$ and $v$ such that $v$ is not adjacent to any endvertex of $e$.
For more details on proper $(2,3)$-poles we refer the reader to  \cite[Section~5.2]{Morphology}.

Next, we create a $(2,2,2)$-pole $H_6$ from a $6$-cycle 
$(u_1u_2u_3v_4u_5u_6)$, encoded as a cyclic sequence of vertices, by attaching a dangling edge $e_i$ to the vertex $u_i$ for each $i\in\{1,\ldots,6\}$. We partition the dangling edges into three connectors $(e_1, e_4)$, $(e_2, e_5)$, $(e_3, e_6)$ and join the resulting $2,2,2)$-pole to $\TTTsc(T_1, T_2, T_3)$. 
Strictly speaking, there are several ways of how to join $H_6$ to $\TTTsc(T_1, T_2, T_3)$, but due to the symmetry of $H_6$, there is essentially only one way that preserves the connectors. All other graphs can be obtained by permuting the multipoles $T_1$, $T_2$, and $T_3$ and by switching the pairs of edges within individual $2$-connectors of $T_1$, $T_2$, and $T_3$. For the purpose of our proofs, we can therefore regard the join $H_6*\TTTsc(T_1, T_2, T_3)$ as uniquely determined.

Note that usually, and also in the next section, the cycles are regarded as $6$-poles with their semiedges ordered in one of the natural cyclic orders. The considered partition of semiedges into three connectors in the $(2,2,2)$-pole $H_6$ is specifically needed for this construction and it also reflects some of the colouring properties that are specific to the $6$-cycle.




\begin{figure}[h]
    \centering
    \includegraphics{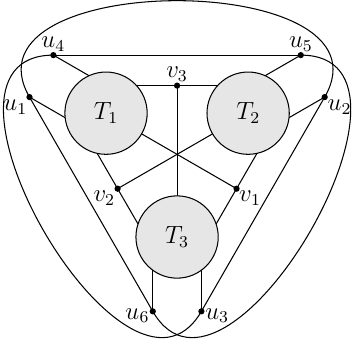}
    \caption{The structure of the snark $H_6*\TTTsc(T_1, T_2, T_3)$}
    \label{fig:ttt-sc-complete}
\end{figure}

\begin{lemma}\label{lemma:ttt-sc-snark}
    For any proper $(2, 3)$-poles $T_1$, $T_2$, and $T_3$ the graph $G = H_6*\TTTsc(T_1, T_2, T_3)$ is a snark in which each pair of vertices $v_1$, $v_2$ and $v_3$ is removable. In particular, $G$ is not bicritical.
\end{lemma}

\begin{proof}
It is well known that $H_6$ is an \emph{even} $(2,2,2)$-pole, which means that the number of proper connectors is always even (see \cite[Section~5.4]{Morphology}). By contrast,
each connector of $\TTTsc(T_1, T_2, T_3)$ is proper, so $H_6*\TTTsc(T_1, T_2, T_3)$ is indeed a snark. Observe that all three $2$-connectors of $\TTTsc(T_1, T_2, T_3)$ will be proper regardless of the presence of the vertices $v_1$, $v_2$ and $v_3$. Thus, any two of them form a removable pair vertices.
\end{proof}

In general, $H_6*\TTTsc(T_1, T_2, T_3)$ need not be a critical snark, and this may occur even in the case where the all three $(2,3)$-poles are taken from a critical snark. The reason is that a proper $(2,3)$-pole  $T= (G - e) - v$ constructed from a critical snark may happen to be uncolourable.
Thus to ensure the criticality we need an additional rather technical property based on the following definition.


%
%

\begin{definition}[Chladný and Škoviera \cite{Chladny-Skoviera-Factorisations}]
     A pair of distinct edges $\{e,f\}$ of a snark $G$ is \textit{essential} if it is non-removable and for every vertex $v$ of the $4$-pole $G - (e, f)$ incident with some dangling edge, the $3$-pole $G \sim v$ is colourable.
\end{definition}

\begin{definition}
	\label{def:good-proper}
	A proper $(2,3)$-pole  $T = (G - e) - v$ is called \emph{good} if, for every endvertex $w$ of the edge $e$ and every pair of edges $f$ and $g$ such that $f$ is incident with $w$ and $g$ is incident with $v$, the pair $\{ f, g \}$ is essential in $G$.
\end{definition}

The following lemma asserts that for certain multipoles, including good proper $(2,3)$-poles, we can find a colouring in which two prescribed semiedges have the same colour.
\begin{lemma}\label{lemma:good-proper}
	Let $G$ be a snark and let $x$ and $y$ be non-adjacent vertices such that for any two edges $e$ and $f$ incident with $x$ and $y$, respectively, the pair $\{e, f\}$ is essential in $G$.
    Then there exists a colouring $\varphi$ of the $6$-pole $M((e_1, e_2, e_3), (f_1, f_2, f_3)) =  G - (x, y)$ such that $\varphi(e_1) = a$, $\varphi(e_2) = \varphi(e_3)$, $\varphi(f_1) = c$ and $\varphi(f_2) + \varphi(f_3) = b$ for some $a$, $b$, $c$ such that $\{a, b, c\} = \K$.
\end{lemma}

\begin{proof}
Since the pair of the edges $\{e_1, f_1\}$ is essential in $G$, the $3$-pole  $(G - (e_1, f_1)) \sim x$ has a colouring from which the desired  colouring of $G - (x, y)$ can be obtained in a straightforward way.
\end{proof}

Next, we recall the following well-known fact on the colouring set of $(2, 2)$-poles obtained by removing a non-removable pair of adjacent vertices from a snark, also called \emph{isochromatic $(2, 2)$-poles}.

\begin{lemma}[{Chladný and Škoviera \cite[Section~3]{Chladny-Skoviera-Factorisations}}]
\label{lemma:isochromatic}
	Let $G$ be a snark and $\{u, v\}$ a pair of adjacent non-removable vertices of $G$. Then $$\col(G - [u, v]) = \{ (a, a, b, b) \mid a, b \in \K \}.$$
\end{lemma}




The following theorem is a crucial step towards the proof of Theorem~\ref{thm:cyc5}. It states that three good $(2,3)$-poles in a $\TTT$-multipole are sufficient to produce a critical snark.

\begin{theorem}\label{thm:good-sc}
	Let $T_1$, $T_2$, and $T_3$ be perfect good proper $(2,3)$-poles obtained from critical snarks. Then the snark $G = H_6 * \TTTsc(T_1, T_2, T_3)$ is strictly critical.
\end{theorem}

\begin{proof}
Assume $T_i = (G_i - e_i) - w_i$ for each $i \in \{1,2,3\}$. Since the snark $G$ is not bicritical by Lemma \ref{lemma:ttt-sc-snark}, it is sufficient to show that it is critical, that is, that for an arbitrary pair of adjacent vertices $x$ and $y$ of $G$ the $4$-pole $M = G - [x,y]$ is colourable. We distinguish four cases.

\paragraph{Case (i).} Both vertices $x$ and $y$ belong to the $6$-cycle in $H_6$.  The $4$-pole $M$ can be coloured as shown in Figure~\ref{fig:ttt-sc-1}. All the $(2,3)$-poles $T_i$ admit colourings as displayed in the figure since they are perfect.

\paragraph{Case (ii).} The vertex $x$ is from the $6$-cycle and (without loss of generality) the vertex $y$ is from $T_1$.
A colouring of the $4$-pole $M$ is indicated in Figure~\ref{fig:ttt-sc-2}. Since $y$ has to be an end vertex of $e_i$, the vertices $y$ and $w_1$ are non-adjacent. So, by Lemma~\ref{lemma:good-proper}, the $6$-pole $N(S_1, S_2) = G_1 - (y, w_1)$ admits a colouring assigning the colours $a$, $p$, $p$ to the semiedges of $S_1$ in the order depicted in Figure~\ref{fig:ttt-sc-2}, and colours $a$, $c$, $c$ to the semiedges of $S_2$ in some order. Due to the symmetry of the multipole $\TTTsc(T_1,T_2,T_3)$, it is irrelevant which semiedge of $S_2$ has colour $a$. 

\begin{figure}[!h]
	\centering
	\begin{minipage}[c]{0.45\textwidth}
		\centering
		\includegraphics{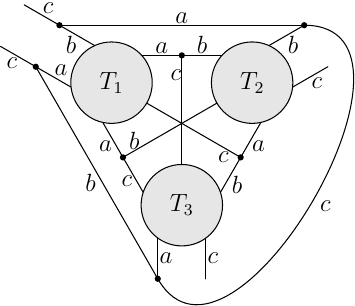}
		\captionof{figure}{Colouring for Case (i)}
		\label{fig:ttt-sc-1}
	\end{minipage}
	\hfill
	\begin{minipage}[c]{0.45\textwidth}
		\centering
		\includegraphics{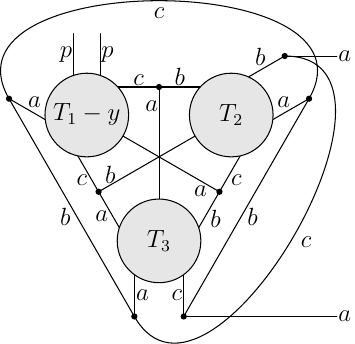}
		\captionof{figure}{Colouring for Case (ii)}
		\label{fig:ttt-sc-2}
	\end{minipage}
\end{figure}
\begin{figure}[h]
	\centering
	\includegraphics{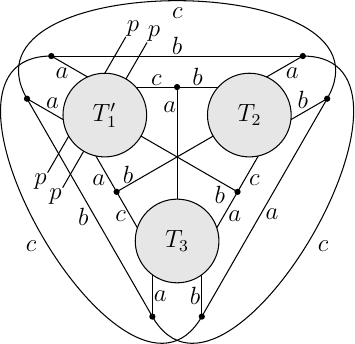}
	\caption{Colouring for Case (iii)}
	\label{fig:ttt-sc-3}
\end{figure}

\paragraph{Case (iii).} Both  $x$ and $y$ belong to the same proper $(2,3)$-pole, say, $T_1$. A colouring of $G - [x,y]$ is illustrated in Figure~\ref{fig:ttt-sc-3}. It includes a colouring of the $9$-pole $T_1' = T_1 - [x, y]$, which can be obtained in the following way. Consider the snark $G_1$ which gives rise to the $(2,3)$-pole $T_1$ by removing an edge $e_1$ and a vertex $w_1$. We start from a colouring of the $4$-pole $G_1 - [x, y]$ such that all the dangling edges have the same colour $p$; such a colouring exists according to Lemma~\ref{lemma:isochromatic}. Denote the colour of the edge $e_1$ of $G_1 - [x,y]$ by $a$. The links incident with the vertex $w_1$ have pairwise distinct colours $a$, $b$, $c$ in some order. By symmetry, the exact order of these colours is not important.
After the removal of the vertex $w_1$ and splitting the link $e_1$, we get the required colouring of $T_1 - [x, y]$.

\paragraph{Case (iv).} The vertex $x$ is from $T_1$ and $y$ is one of $v_1$, $v_2$ and $v_3$, say $y = v_1$.
Initially, we find a colouring of the $6$-pole $T_1 - x \cong (G_1 - [w_1, x]) - e_1$. Since the snark $G_1$ is critical, the $4$-pole $G_1 - [w_1, x]$ is colourable in such a way that all its dangling edges are coloured by the same colour $c$ (by Lemma~\ref{lemma:isochromatic}). Denote the colour of the link $e_1$ of $G_1 - [w_1, x]$ by $p$. After splitting the link $e_1$, we obtain the $6$-pole $T_1 - x \cong (G_1 - [w_1, x]) - e_1$ with its semiedges coloured by $c$, $c$, $c$, $c$, $p$, $p$. If $p \ne c$, the colouring of the $4$-pole $M = G - [x,y]$ is given in Figure~\ref{fig:ttt-sc-4a} (with $p=a$). Otherwise, if $p = c$, we can colour the $4$-pole $M$ according to Figure~\ref{fig:ttt-sc-4b}.
\end{proof}

\begin{figure}[!h]
	\centering
	\begin{subfigure}[c]{0.45\textwidth}
		\centering
		\includegraphics{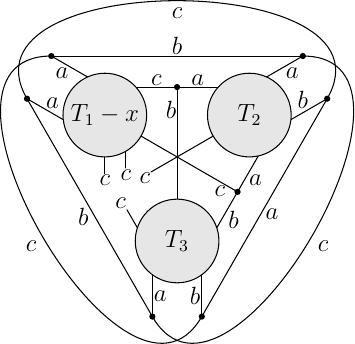}
		\caption{$p = a \ne c$}
		\label{fig:ttt-sc-4a}
	\end{subfigure}
	\hfill
	\begin{subfigure}[c]{0.45\textwidth}
		\centering
		\includegraphics{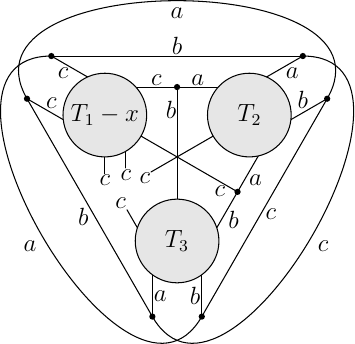}
		\caption{$p = c$}
		\label{fig:ttt-sc-4b}
	\end{subfigure}
	\caption{Colouring for Case (iv)}
\end{figure}

Now we are ready to prove Theorem~\ref{thm:cyc5}.

\begin{proof}[Proof of Theorem~\ref{thm:cyc5}]
For the proof, we take all three proper $(2, 3)$-poles $T_1$, $T_2$ and $T_3$
from the well-known family of Isaacs snarks. It is well known that Isaacs snarks are critical excluding $J_3$, which is not critical \cite[Theorem~4.13]{SteffenClassifications}. If $T_i$ is to be obtained from $J_5$, we choose an edge and a vertex so that their removal destroys the only $5$-cycle of $J_5$. As a result, each $T_i$ has girth $6$. Since every pair of non-adjacent edges in any Isaacs snark except $J_3$ is essential \cite[Example 5.5]{Chladny-Skoviera-Factorisations}, the proper $(2, 3)$-poles $T_1$, $T_2$ and $T_3$ are all good. Theorem~\ref{thm:good-sc} now implies that the graph $H_6*\TTTsc(T_1, T_2, T_3)$ is a strictly critical snark. Moreover, it is clear that it is cyclically $5$-connected and has girth $6$.

The smallest member of our family has $66$ vertices and is depicted in Figure~\ref{fig:ttt-j5}; it is produced by taking each proper $(2, 3)$-pole $T_i$ from the Isaacs snark $J_5$ as described above. In order to construct larger snarks, we can take the snark on $66$ vertices and replace one of the $(2,3)$-poles constructed from $J_5$ with a ($2,3)$-pole constructed from $J_{5+2k}$ where $k$ is an arbitrary positive integer, thereby adding $8k$ vertices. This completes the proof of Theorem~\ref{thm:cyc5}.
\end{proof}

\begin{remark}
{\rm We are not aware of any smaller strictly critical snarks without $5$-cycles. } 
\end{remark}


\begin{figure}
	\centering\includegraphics[]{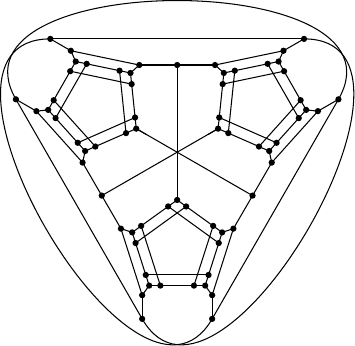}
	\caption{A cyclically 5-connected strictly critical snark of girth $6$ and order $66$}
	\label{fig:ttt-j5}
\end{figure}

\section{Cyclically 6-connected strictly critical snarks}
\label{sec:sc-6-6}

In this section we present a construction of cyclically $6$-connected strictly critical snarks based on superposition. Superposition is a  general construction method producing large cubic graphs (usually snarks) from smaller ones. Here, we restrict ourselves to basic facts about this method, which are  necessary for our exposition, and refer the reader to \cite{Superposition, MS:superp} for more details.

In order to define a superposition, we first choose a \emph{base graph} $G$, which we assume to be cubic and connected. Then we replace each vertex with a~\emph{supervertex}, a multipole with tree connectors, and each edge with a \emph{superedge}, a multipole with two connectors. Finally, we perform all the
junctions between supervertices and superedges that correspond to the incidences between vertices and edges of $G$. The corresponding connectors are required to have the same width in order for the resulting graph $\tilde G$ be cubic. 

There are multiple ways how one can ensure that a superposition produces a snark. According to \cite[Theorem~4]{Superposition}, if the base graph is a snark and each of the superedges used in the superposition is proper, then the graph resulting from superposition is again a snark. Such a superposition is called a \emph{proper superposition}.


In our construction we utilise a specific type of the junction of multipoles. Take two multipoles with three connectors each, an $(i,s,r_1)$-pole $M(I, S_1, R_1)$ and an $(s,o,r_2)$-pole $N(S_2, O, R_2)$ such that $|S_1|=|S_2|$. Following \cite{Macajova}, we define a \emph{serial junction} $M \circ N$ of $M$ and $N$ as the $(i,o,r_1+r_2)$-pole $P(I, O, R_1 \cup R_2)$ which arises by the junction of the connectors $S_1$ and $S_2$ and by merging the connectors $R_1$ and $R_2$ into a single connector $R_1 \cup R_2$. The connectors $R_1$ and $R_2$ are regarded as \emph{residual}, that is, not involved in the junction. We allow $r_1$ or $r_2$ to be zero and treat a multipole $M(S_1, S_2, \emptyset)$ with an empty residual connector as $M(S_1, S_2)$. Given a multipole $M(S_1, S_2, \dots, S_k)$ with $k\ge 2$ and $|S_1| = |S_2|$, we also define the \emph{closure} $\overline{M}$ of $M$ to be the multipole obtained by the junction of $S_1$ and $S_2$.

\begin{figure}[!h]
	\centering
	\begin{minipage}[c]{0.4\textwidth}
		\centering
		\includegraphics{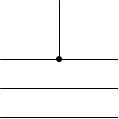}
		\captionof{figure}{The supervertex $W$}
		\label{fig:supervertex}
	\end{minipage}
	\begin{minipage}[c]{0.5\textwidth}
		\centering
		\includegraphics{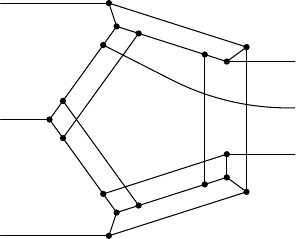}
		\captionof{figure}{The superedge $A_5$}
		\label{fig:superedge}
	\end{minipage}
\end{figure}

We now proceed to our particular superposition construction. 
For supervertices we will take copies of the multipole $W$ depicted in Figure~\ref{fig:supervertex}, which has two $3$-connectors and a $1$-connector in this order.
Our superedges will be created from Isaacs snarks.
For a detailed explanation of properties of Isaacs snarks and relevant notation we refer the reader to \cite[Section 5.5]{Morphology}. 
The \emph{Isaacs superedge} $A_5$ is obtained by removing two specific vertices at distance $4$ from the Isaacs snark $J_5$ as shown in Figure~\ref{fig:superedge}. For any odd $k \ge 7$, the \emph{Isaacs superedge} $A_k$ is constructed by substituting the $6$-pole $Y_{k - 3}$ for the $6$-pole $Y_2$ contained in $A_5$. Note that $A_k$ also arises from the Isaacs snark $J_k$ by removing two vertices at  distance $4$. It is known that all the Isaacs superedges $A_k$ are proper $(3,3)$-poles \cite{Superposition, Morphology}.

We will need the following two properties of multipoles constructed from Isaacs snarks.

\begin{lemma}
    \label{lemma:perfect23}
    Let $J_{k}$ be the Isaacs snark, where $k \ge 5$ is odd, and let
    $u$ be a vertex and $vw$ be an edge of $J_{k}$ such that
    $u$ is adjacent to neither $v$ nor $w$.
    Then the proper $(2,3)$-pole $(J_{k} - vw) - u$ is perfect.
\end{lemma}

\begin{proof}
We proceed by induction. The statement holds for $J_5$ and $J_7$, which we have verified by using a computer. Assume that the statement is true for some odd $k \ge 7$. Then the proper $(2,3)$-pole $T = (J_{k+2} - e) - v$ contains a $6$-pole $Y_4$ which is colour-equivalent to the $6$-pole $Y_2$. Therefore, the proper $(2,3)$-pole $T$ is colour-equivalent to some proper $(2,3)$-pole constructed from $J_k$, which is perfect by the induction hypothesis. 
\end{proof}

As a straightforward consequence, we have the following lemma.

\begin{lemma}
    \label{lemma:isaacs-33-pole}
    Let $M = J_k - (u, v)$ be a $(3,3)$-pole constructed from an Isaacs snark $J_k$ for some odd $k \ge 5$ by removing vertices $u$ and $v$ whose distance is at least $3$ and let $a, b, c, d, e$ be any colours from $\K$ such that $a + b = c + d + e \ne 0$. Then
    $$(a, b_2, b_3, c, d, e) \in \col(M)$$
    for some $b_2, b_3 \in \K$ with $b_2 + b_3 = b$.
\end{lemma}

\begin{proof}
    Let $i_1, i_2, i_3$ be the links incident with $u$ in $J_k$ and $o_1, o_2, o_3$ be the links incident with $v$. Since the proper $(2,3)$-pole $N((i_1, j_2), O) = J_k - i_1 - v$ is perfect according to Lemma \ref{lemma:perfect23}, it admits a colouring $\varphi$ with $\varphi(S(N)) = (a, b, c, d, e)$. Then $\varphi(i_2) = b_2$ and $\varphi(i_3) = b_3$. After removing the vertex $u$ together with the dangling edge $j_2$ we obtain the $(3,3)$-pole $M$ with the desired colouring.
\end{proof}

The key idea of our superposition is to choose a cycle of the base graph and replace its vertices and edges with the just introduced supervertices and superedges thereby producing a larger multipole called a supercycle.
First, we recursively define a~\emph{$k$-superpath} as follows: A $1$-superpath is any $(3,3,1)$-pole of the form ${F_1} \circ W \circ {F_2}$, where $F_1$ and $F_2$ are arbitrary Isaacs superedges. For $k\ge 2$, a $k$-superpath is a $(3,3,k)$-pole of the form $SP_{k-1} \circ W \circ {F_{k+1}}$ for an arbitrary Isaacs superedge $F_{k+1}$ and an arbitrary $(k-1$)-superpath $SP_{k-1}$. For $k \ge 2$, a \emph{$k$-supercycle} $SC_k$ is a $k$-pole of the form $\overline{SP_{k-1} \circ W}$, where $SP_{k-1}$ is an arbitrary $(k - 1)$-superpath.

Supercycles and superpaths have colouring properties similar to those of cycles and paths, respectively. At this point it may be convenient to regard paths and cycles as multipoles with a dangling edge attached to each vertex. To be more precise, for $k \ge 1$, the \emph{$k$-cycle} is the $k$-pole $C_k(e_1, e_2, \dots, e_k)$ consisting of a cycle $(v_1v_2\dots v_k)$ with a dangling edge $e_i$ attached to $v_i$ for each $i\in\{1, \dots, k\}$. 
For $k \ge 1$, the \emph{$k$-path} is the $(1,1,k)$-pole $P_k((i), (o), (r_1, r_2, \dots, r_{k}))$ whose underlying graph is a path $v_1\dots v_k$, the dangling edges $i$ and $o$  are incident with vertices $v_1$ and $v_k$, respectively, and the dangling edge $r_i$ is incident with $v_i$ for each $i\in\{1, \dots, k\}$. Note that the $k$-path contains $k$ vertices and $k - 1$ links.

\begin{lemma}
	\label{lemma:superpath}
    Let $k \ge 1$ and let $SP_k$ be an arbitrary $k$-superpath. Consider an arbitrary element
    $$
        (c_i, c_o, c_1, \dots, c_k)\in \col(P_{k})
    $$
    for a $k$-path $P_k$ and arbitrary elements $a_1$, $a_2$, $a_3$, $b_1$, $b_2$, $b_3\in \K$ such that $a_1 + a_2 + a_3 = c_i$ and $b_1 + b_2 + b_3 = c_o$. Then 
    $$
        (a_1, a_2, a_3, b_1, b_2, b_3, c_1, \dots, c_k)\in \col(SP_k).
    $$
\end{lemma}

\begin{proof}
    We proceed by induction on $k$.
    For $k = 1$, we have verified that the $(3,3,1)$-pole $A_5 \circ W \circ A_5$ has the desired colouring set by using a computer. Every other superedge $A_\ell$ can be obtained from $A_5$ by substituting the $(3,3)$-pole $Y_{\ell - 3}$ for  the only copy of $Y_2$ contained in $A_5$. Since $\col(Y_{\ell - 3})=\col(Y_2)$, we have $\col(SP_1) = \col(A_x \circ W \circ A_y)$ for any odd $x, y \ge 5$.
	

    Now, assume that the statement holds for any $(k - 1)$-superpath. Consider a $k$-superpath $SP_k = {F_1} \circ W_1 \circ {F_2} \circ \dots \circ W_k \circ F_{k+1}$ for some $k \ge 2$.
    Let $(c_i, c_o, c_1, \dots, c_k) \in \col(P_k)$. Consider arbitrary elements $a_1,\, a_2,\, a_3,\, b_1,\, b_2,\, b_3 \in \K$ such that $a_1 + a_2 + a_3 = c_i$ and $b_1 + b_2 + b_3 = c_o$.
    Choose $d \in \K - \{c_1\}$.
    By Lemma \ref{lemma:isaacs-33-pole}, we have $(a_1, a_2, a_3, d, d_1, d_2) \in \col(F_{1})$ for some $d_1, d_2 \in \K$ satisfying Parity Lemma. Then clearly, $(d, d_1, d_2, d + c_1, d_1, d_2, c_1) \in \col(W_1)$.
    Finally, by the induction hypothesis, we have $(d + c_1, d_1, d_2, b_1, b_2, b_3, c_2, c_3, \dots, c_k) \in \col({F_2} \circ W_2 \circ {F_3} \circ \dots \circ W_k \circ F_{k+1})$ which yields the desired colouring.
\end{proof}

\begin{lemma}
	\label{lemma:supercycle}
If $SC_k$ is an arbitrary supercycle with $k \ge 2$, then $\col(SC_k) = \col(C_k)$.
\end{lemma}

\begin{proof}
Because our superposition is proper, we conclude that $\col(SC_k) \subseteq \col(C_k)$. To prove the converse, let $SC_k = \overline{SP_{k-1} \circ W}$, and consider a $k$-cycle $C_k(e_1, e_2, \dots, e_k)$; denote by $v_i$ the vertex incident with $e_i$ for each $i \in \{1, 2, \dots, k\}$. Take a colouring $\varphi$ of $C_k$ and let $c_i = \varphi(e_i)$ for each $i \in \{1, 2, \dots, k\}$. The $(3,3,1)$-pole $W$ is colourable in such a way that the flows through its connectors are $a = \varphi(v_{k-1}v_k)$, $b = \varphi(v_kv_1)$ and $c_k$, respectively. Since $(b, a, c_1, c_2, \dots, c_{k-1}) \in \col(P_{k-1})$ according to Lemma~\ref{lemma:superpath}, the $(k - 1)$-superpath $SP_{k-1}$ admits desired colours on its semiedges, and therefore $\varphi(SC_k) \in \col(SC_k)$.
\end{proof}


In all our superpositions, we substitute a $k$-supercycle $SC_k$ for a copy $C$ of the $k$-cycle $C_k$ in a snark $G$. The resulting snark will be  denoted by $\superp(G, C, SC_k)$.

\bigskip


\begin{theorem}
\label{thm:superposition}
If $G$ is a critical snark, then the graph $\tilde G=\superp(G, C, SC_k)$ is also a critical snark.
\end{theorem}

\begin{proof}
We start the proof by introducing some useful notation. Let $G = M*C$ be a critical snark containing a $k$-cycle $C(e_1, e_2, \dots, e_k)$; for each $i\in \{1, 2, \dots k\}$ let $v_i$ denote the vertex incident with the dangling edge $e_i$. Let $SC_k$ be a $k$-supercycle such that
$$
SC_k(r_1, r_2, \dots, r_k) = \overline{F_1 \circ W_1 \circ F_2 \circ \dots \circ W_{k-1} \circ F_k \circ W_k}
$$
where, for each $j\in \{1, 2, \dots k\}$, $W_j$ is a copy of the supervertex $W$, $w_j$ is the only vertex of $W_j$, and
$$
{F_j}((i_j^1, i_j^2, i_j^3), (o_j^1, o_j^2, o_j^3)) = J_{\ell_j} - (u_j, t_j)
$$
for some odd $\ell_j \ge 5$. The index $j$ in all the places where we use it is taken modulo $k$. The notation for semiedges $i_j^1,
\dots, o^3_j$ is also used for their original edges (that is, the links of $J_{\ell_j}$ from which the semiedges arise).

    \begin{figure}[h]
		\centering
		\includegraphics[]{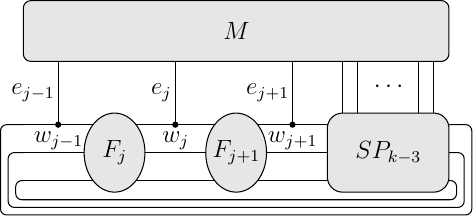}
		\caption{Structure of the snark $\tilde G$ used throughout the proof}
		\label{fig:superposition-0}
	\end{figure}

Let $\tilde G = M*SC_k = \superp(G, C, SC_k)$. A schematic drawing of $\tilde G$ is given in Figure~\ref{fig:superposition-0}, where two consecutive superedges $F_j$ and $F_{j+1}$ from the supercycle $SC_k$ are depicted separately. The remaining $(k - 3)$-superpath $F_{j+2} \circ W_{j+2} \circ F_{j+3} \circ \dots \circ W_{j-2} \circ F_{j-1}$  is denoted by $SP_{k-3}$ (note that $F_{j-1} = F_{j+2+k-3}$).

Since $\tilde G$ arises from $G$ by a proper superposition, $\tilde G$ is a snark. It remains to show that $\tilde G$ is critical.
Let $x$ and $y$ be arbitrary adjacent vertices in the snark $\tilde G$. We need to show that the $4$-pole $\tilde G - [x, y]$ is colourable. The vertices $x$ and $y$ might come from the $k$-pole~$M$, from a superedge contained in $SC_k$, or from a supervertex contained in $SC_k$. Accordingly, the proof splits into several cases.
	
\paragraph{Case (i).} Both $x$ and $y$ belong to the $k$-pole $M$. Since $G$ is critical, the $4$-pole $M*C - [x, y]$ is colourable. The colourability of $M*SC_k - [x, y]$ then follows from the fact that $C$ and $SC_k$ are colour-equivalent according to Lemma~\ref{lemma:supercycle}.

\paragraph{Case (ii).} The vertex $x$ belongs to $M$ and $y = w_j$. for some $j$. By Lemma~\ref{lemma:isochromatic}, the $(2,2)$-pole $N((f_1, f_2), (f_3, f_4)) = G - [x, v_j]$ admits a colouring $\varphi$ such that $\varphi(f_1) = \varphi(f_2) = \varphi(f_3) = \varphi(f_4) = a$. For this colouring, we have $\varphi(v_{j-1}v_{j-2}) = p$, $\varphi(e_{j-1}) = p + a$, $\varphi(v_{j+1}v_{j+2}) = q$, and $\varphi(e_{j+1}) = q + a$ for some (not necessary distinct) elements $p,\, q \in \K-\{a\}$; note that $p$ and $q$ must be different from $a$, for otherwise the colour of some edge would be zero. We colour the $(k+1)$-pole $SC_k - y$ as shown in Figure~\ref{fig:superposition-1}; the colour of $r_i$ is set to $\varphi(e_i)$ for all $i \ne j$. From Lemmas \ref{lemma:isaacs-33-pole} and \ref{lemma:superpath} infer that the superedges ${F_j}$, ${F_{j + 1}}$, and the superpath $SP_{k-3}$ admit such colourings.
	
	\begin{figure}[h!]
		\centering
		\includegraphics{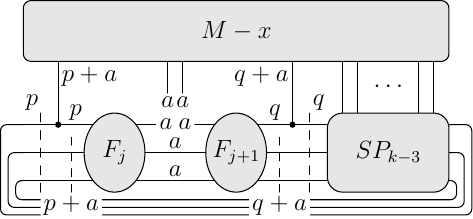}
		\caption{Colouring for Case (ii)}
		\label{fig:superposition-1}
	\end{figure}
	
\paragraph{Case (iii).} The vertex $x$ is identical with $w_j$, for some $j$, and $y$ belongs to an adjacent superedge, say ${F_j}$. Lemma~\ref{lemma:isochromatic} now implies that the $(2,2)$-pole $N_1 = J_{\ell_j} - [y, t_j]$ can be coloured in such a way that all its dangling edges
receive the same colour $q$. Hence, if we remove the vertex $u_j$ from the $(2,2)$-pole $N_1$ we get the $(3,2,2)$-pole $F_j'=(J_{\ell_j} - u_j) - [y, t_j]$ with its dangling edges coloured $(a,b,c,q,q,q,q)$, where $\{a, b, c\} = \K$.

Since $G$ is a critical snark, Lemma~\ref{lemma:isochromatic} implies that the $(2,2)$-pole $N((f_1, f_2), (f_3, f_4)) = G - [v_{j-1}, v_j]$ admits a colouring $\varphi$ such that $\varphi(f_1) = \varphi(f_2) = p \ne a$ and $\varphi(f_3) = \varphi(f_4) = q$ for some $p, q \in \K$, not necessary distinct. Observe that $\varphi(e_{j+1})=r\ne q$
because $e_{j+1}$ and the semiedge $f_4$ are both incident with $v_{j+1}$ in $N$.
The colouring of the $(2,2)$-pole $M - [x,y]$ is depicted in Figure~\ref{fig:superposition-3};  suitable colourings of $F_{j+1}$ and $SP_{k-2}$ exist according to Lemmas \ref{lemma:isaacs-33-pole} and \ref{lemma:superpath}, respectively.

	\begin{figure}[h!]
		\centering
		\includegraphics{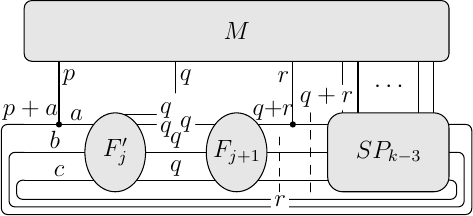}
		\caption{Colouring for Case (iii)}
		\label{fig:superposition-3}
	\end{figure}
	
\paragraph{Case (iv).} Both $x$ and $y$ belong to the same superedge ${F_j}$ for some $j$. From any colouring of the $(2,2)$-pole $J_{\ell_j} - [x, y]$ we can obtain a colouring of the $(3,3,2,2)$-pole $F_j' = (J_{\ell_j} - (u_j, t_j)) - [x, y]$ assigning to its first six dangling edges colours $a, b, c, p, q, r$ (in this order) such that $a + b + c = p + q + r = 0$. According to Lemma~\ref{lemma:isochromatic}, the $(2,2)$-pole $G - [v_j, v_{j+1}]$ admits a colouring $(a', a', p', p')$ for some $a' \ne a$ and $p' \ne p$, therefore we can colour the $(2,2)$-pole $\tilde G - [x,y]$ as shown in Figure~\ref{fig:superposition-4}. A colouring of $F_{j+1} \circ {W_{j+1} \circ SP_{k-3}}$, which is a $(k - 2)$-superpath, can be obtained from Lemma~\ref{lemma:superpath}.

	\begin{figure}[h!]
		\centering
		\includegraphics{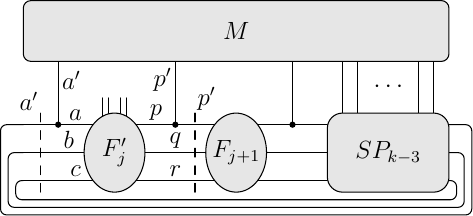}
		\caption{Colouring for Case (iv)}
		\label{fig:superposition-4}
	\end{figure}
	
\paragraph{Case (v).} The vertices $x$ and $y$ belong to two different superedges, say, $x \in V({F_j})$ and $y \in V({F_{j+1}})$. Set $F_j' = F_j-x$ (we discard the dangling edge adjacent to $x$ in $F_j'$) and similarly $F_{j+1}'=F_{j+1}-y$ (see Figure~\ref{fig:superposition-5}). By Lemma~\ref{lemma:isochromatic}, the exists a colouring of the $(2,2)$-pole $N_1 = G - [v_j, v_{j+1}]$ that assigns the same colour $a$ to all its semiedges. The link $e_{j-1}$ of $N_1$ receives a colour $b\ne a$, because it is incident with the dangling edge incident to $v_{j-1}$, which is coloured $a$. 

We proceed by finding suitable colourings for $F_j'$ and $F_{j+1}'$.  First, we colour $F_j'$. By the construction of the Isaacs superedge $F_j$, the vertices $u_j$ and $t_j$ are at distance at least~$4$ in $J_{\ell_j}$, so the distance between $u_j$ and $x$ is at least~$3$. Thus, according to Lemma~\ref{lemma:isaacs-33-pole}, the $(3,3)$-pole $N_2((i_j^1, i_j^2, i_j^3), (f_1, f_2, f_3)) = J_{\ell_j} - (u_j, x)$ admits a colouring $\varphi$ such that $\varphi(i^j_1) + \varphi(i^j_2) + \varphi(i^j_3) = a$, $\varphi(f_1) = a = \varphi(o_1^j) + \varphi(o_2^j)$ and $\varphi(f_2) + \varphi(f_3) = 0$. If we remove the vertex $t_j$ along with its dangling edge $f_1$ from $N_2$, we get two more dangling edges $o_j^1$ and $o_j^2$ which receive colours $b'$ and $c'$, respectively, with $b' + c' = a$. 
    
Next, we find a suitable colouring for $F_{j+1}'$. The $(2,2)$-pole $N_3((i_{j+1}^1, i_{j+1}^2), (g_3, g_4)) = J_{\ell_{j+1}} - [u_{j+1}, y]$ admits a colouring $\psi$ such that $\psi(i_{j+1}^1) = \psi(i_{j+1}^2) = c'$ and $\psi(g_3) = \psi(g_4)$. Under the colouring $\psi$, the links $o_{j+1}^1$, $o_{j+1}^2$ and $o_{j+1}^3$ receive pairwise distinct colours $p$, $q$, and $r$, respectively.
We may assume that $p \ne a$ because if we had $\psi(o_{j+1}^1) = a$, then we could swap the colours $a$, $b'$ in $\psi$ and get $\psi(o_{j+1}^1) \ne a$.

By using all the described colourings and Lemma~\ref{lemma:superpath}, the $(2,2)$-pole $\tilde G - [x, y]$ can be coloured as shown in Figure~\ref{fig:superposition-5}.  This completes Case~(v) as well as the entire proof.
\end{proof}
	
	\begin{figure}[h!]
		\centering
		\includegraphics{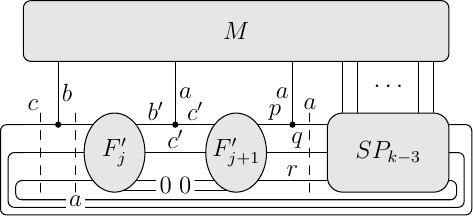}
		\caption{Colouring for Case (v)}
		\label{fig:superposition-5}
	\end{figure}

We can apply Theorem \ref{thm:superposition} to constructing cyclically $6$-connected critical snarks. For strictly critical snarks, however, we need the following lemma.

\begin{lemma}\label{lemma:removable}
Let $G$ be a snark and let $\tilde G$ be a snark obtained from $G$ by a proper superposition which replaces each vertex $v$ of $G$ with a supervertex $V_v$ and each edge of $G$ with a proper superedge $E_e$. If $\{u, v\}$ is a removable pair of vertices of $G$, then for each vertex $\tilde u$ from $V_u$ and each vertex $\tilde v$ from $V_v$ the pair $\{\tilde u, \tilde v\}$ forms a removable pair of vertices of~$\tilde G$.
\end{lemma}

\begin{proof}
Suppose to the contrary that the $6$-pole 
$\tilde G - (\tilde u, \tilde v)$ admits a colouring $\tilde\phi$. 
Since the superedge $E_e$ is proper for each edge $e \in E(G)$, both connectors of $E_e$ have the same nonzero total flow under $\tilde\phi$; we denote this common value by $\varphi(e)$. The assignment $e\mapsto\phi(e)$ defines a mapping $\phi\colon E(G) \rightarrow \K$.
For each vertex $w \in V - \{u, v\}$, the values of $\varphi$ on the edges around $w$ sum to zero due to Parity Lemma applied to $V_w$. The mapping $\varphi$ thus yields a $3$-edge-colouring of the $6$-pole $G - (u, v)$, contradicting the fact that $\{u, v\}$ was a removable pair of vertices of $G$.
\end{proof}

\begin{corollary}
\label{lemma:strictlycritical}
If $G$ is a strictly critical snark with a $k$-cycle $C$, then the snark $\tilde G=\superp(G, C, SC_k)$ is also strictly critical.
\end{corollary}

\begin{proof}
The snark $\tilde G = \superp(G, C, SC_k)$ is critical due to Theorem \ref{thm:superposition}. Since $G$ is not bicritical, it has a pair of non-adjacent vertices whose removal leaves an uncolourable graph. According to Lemma~\ref{lemma:removable}, this pair yields a pair of removable vertices also in $G'$, thus it cannot be bicritical.
\end{proof}



We are now ready to prove our main result, which implies Theorem~\ref{thm:main}.

\begin{theorem}\label{thm:sc-cc6}
There exists a cyclically $6$-connected strictly critical snark of order $n$ in each of the following cases.	
	\begin{enumerate}
	    \item[{\rm (i)}]   $n \equiv 0 \pmod{8}$ and $n \ge 320$.
	    \item[{\rm (ii)}]  $n \equiv 2 \pmod{8}$ and $n \ge 306$.
	    \item[{\rm (iii)}] $n \equiv 4 \pmod{8}$ and $n \ge 324$.
	    \item[{\rm (iv)}]  $n \equiv 6 \pmod{8}$ and $n \ge 342$.
	\end{enumerate}
\end{theorem}

\begin{proof}
First, we construct the smallest snark for each of the 
residue classes$\pmod8$
by taking a suitable strictly critical snark $G$ and replacing a suitable cycle in $G$ with a supercycle. By Corollary \ref{lemma:strictlycritical}, the resulting snark is also strictly critical. The four base snarks used in the constructions for each even residue class$\pmod8$ are depicted in Figure~\ref{fig:sc6} together with the cycle that we replace. In each case, all the superedges used for the corresponding supercycle are isomorphic to $A_5$, which has $18$ vertices.

To prove Item~(i), we start with one of the smallest strictly critical snarks $G_{32}$ (described also in \cite{Chladny-Skoviera-Factorisations}) which is a dot product of an I-extension of one of the two Goldblerg-Loupekine snark of order $22$ and the Petersen graph (see Figure~\ref{fig:sc6-320}). We choose the $16$-cycle indicated in Figure~\ref{fig:sc6-320} and after replacing it with the $16$-supercycle containing superedges isomorphic to $A_5$ we obtain a strictly critical snark $G_{320}$ of order $320$. The snark $G_{32}$ contains one pair $\{x, y\}$ of non-adjacent removable vertices and these two vertices remain removable in the snark $G_{320}$ according to Lemma~\ref{lemma:removable}.

For the remaining three items, we take the snark $G_{36}$ depicted in Figures \ref{fig:sc6-306}, \ref{fig:sc6-324}, and \ref{fig:sc6-342} which is of the form $H_6 * \TTTsc(T_P, T_P, T_P)$, where $T_P$ denotes the proper $(2,3)$-pole constructed from the Petersen graph (see Section \ref{sec:cc5} for the definitions of $H_6$ and $\TTTsc$). If we replace the $15$-cycle indicated in Figure~\ref{fig:sc6-306}, we obtain a strictly critical snark $G_{306}$ of order $306$. Replacing the $16$-cycle indicated in Figure~\ref{fig:sc6-324} yields $G_{324}$ (of order $324$) and replacing the $17$-cycle in Figure~\ref{fig:sc6-342} yields $G_{342}$ (order $342$).
    
Let $\tilde v$ denote the only vertex contained in the supervertex that replaces a vertex $v$. In $G_{36}$, every $2$-element subset of $\{x, y, z\}$ (see Figure~\ref{fig:sc6}) is a pair of non-adjacent removable vertices, thus according to Lemma~\ref{lemma:removable}, also every pair of vertices from the set $\{ \tilde x, \tilde y, \tilde z \}$ is removable in $G_{306}$, $G_{324}$, and $G_{342}$, respectively.

    \begin{figure}[h]
    	\begin{subfigure}{0.45\textwidth}
	    	\centering
	    	\includegraphics[]{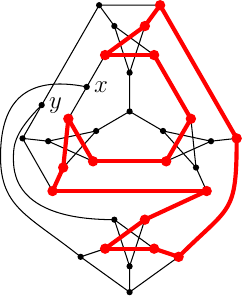}
			\caption{}
			\label{fig:sc6-320}
    	\end{subfigure}
    	\begin{subfigure}{0.45\textwidth}
	    	\centering
	    	\includegraphics[]{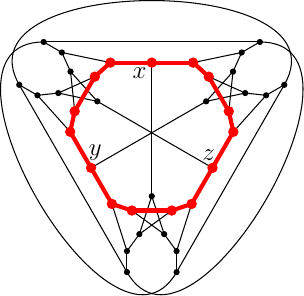}
			\caption{}
			\label{fig:sc6-306}
    	\end{subfigure}
    	
    	\medskip
    	\begin{subfigure}{0.45\textwidth}
    	    		\centering
    	    		\includegraphics[]{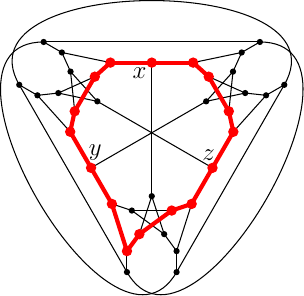}
    	    		\caption{}
    	    		\label{fig:sc6-324}
    	    	\end{subfigure}
    	    	\begin{subfigure}{0.45\textwidth}
    	    		\centering
    	    		\includegraphics[]{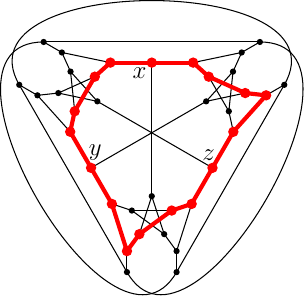}
    	    		\caption{}
    	    		\label{fig:sc6-342}
    	    	\end{subfigure}
    	\caption{The snarks $G_{32}$ and $G_{36}$ with cycles to replace in superposition}
    	\label{fig:sc6}
    \end{figure}

For each $n \in \{320, 306, 324, 342\}$, the snark $G_n$ is strictly critical, has order $n$ and, as we have verified by using a computer program, it is cyclically $6$-connected. If we replace one of the superedges $A_5$ contained in $G_n$ with the superedge $A_{5 + 2k}$ for an integer $k \ge 1$, we obtain a snark $G_{n + 8k}$ of order $n + 8k$. The snark $G_{n + 8k}$ is strictly critical in view of Corollary \ref{lemma:strictlycritical}. It is cyclically $6$-connected: the larger constructed snarks arise from the smallest ones by replacing $Y_2$ with $Y_{2m}$ for some $m$, and such a substitution maintains cyclic connectivity 6 in the same way as it does for the Isaacs snarks. Hence, the snarks $G_n$ constructed above possess all the required properties, and the proof is complete.
\end{proof}

To our best knowledge, the snark $G_{306}$ is the smallest known cyclically $6$-connected strictly critical snark. It is hard to believe that there are no smaller such snarks, but discovering them would very likely require a new method of constructing cyclically $6$-connected snarks, since superposition tends to produce large graphs. This suggests the following problem.

\begin{problem}\label{problem1}
Does there exist a cyclically $6$-connected strictly critical snark of order smaller than $306$?
\end{problem}

We believe that a statement about bicriticality analogous to Theorem~\ref{thm:superposition} might also be true, but the proof seems too complicated if we use our current method.

\begin{problem}
Is it true that for every bicritical snark $G$ any $k$-cycle $C$ in $G$, the snark $\superp(G, C, SC_k)$ is always bicritical?
\end{problem}

If the answer was yes (perhaps with some additional assumptions about cycles being replaced), it would provide a useful tool for constructions of cyclically $6$-connected bicritical snarks.

\subsection*{Acknowledgements}

The authors acknowledge partial support from the research grants APVV-19-0308, VEGA 1/0743/21 and VEGA 1/0727/22.

\end{document}